\documentclass[12pt,a4paper]{amsart}

\usepackage{amsfonts}
\usepackage{amsthm}
\usepackage{amsmath}
\usepackage{amscd}
\usepackage[latin2]{inputenc}
\usepackage{t1enc}
\usepackage[mathscr]{eucal}
\usepackage{indentfirst}
\usepackage{graphicx}
\usepackage{graphics}
\usepackage{pict2e}
\usepackage{epic}
\numberwithin{equation}{section}
\usepackage[margin=2.9cm]{geometry}
\usepackage{epstopdf} 
\usepackage{times}
\usepackage{enumerate}
\usepackage{url}
\usepackage{changepage}
\usepackage{amssymb}
\usepackage{mfirstuc}
\usepackage{tikz}
\usepackage{cite}

\theoremstyle{plain}
\newtheorem{Th}{Theorem}[section]
\newtheorem{Lemma}[Th]{Lemma}
\newtheorem{Conj}[Th]{Conjecture}

\newtheorem{Prop}[Th]{Proposition}
\newtheorem{Def}[Th]{Definition}

 \theoremstyle{definition}

\newtheorem{?}[Th]{Problem}
\newtheorem{Ex}[Th]{Example}

\newcommand{\Q}{\mathbb{Q}}
\newcommand{\R}{\mathbb{R}}

\newcommand{\C}{\mathbb{C}}
\newcommand{\Z}{\mathbb{Z}}
\newcommand{\F}{\mathbb{F}}
\newcommand{\ord}{\mathrm{ord}}

\begin{document}

\title{Hasse Polynomials of L-functions of Certain Exponential Sums}

\author{Chao Chen}
\address{Department of Mathematics, University of California, Irvine, 
Irvine, CA 92697-3875, USA}
\email{chaoc12@uci.edu}



\begin{abstract}In this paper, we focus on computing the higher slope Hasse polynomials of L-functions of certain exponential sums associated to the following family of Laurent polynomials $f(x_1,\ldots ,x_{n+1})=\sum_{i=1}^na_i x_{n+1}\left(x_i+\frac{1}{x_i}\right)+a_{n+1} x_{n+1}+\frac{1}{x_{n+1}}$, where $a_i \in \F^*_{q}$, $i=1,2, \ldots, n+1$. We find a simple formula for the Hasse polynomial of the slope one side and study the irreducibility of these Hasse polynomials. We will also provide a simple form of all the higher slope Hasse polynomials for $n=3$, answering an open question of Zhang and Feng.   
\end{abstract}

\maketitle

\section{Introduction} 
Let $\F_{q}$ be the finite field of $q$ elements with characteristic $p$. For each positive integer $k$, let $\mathbb{F}_{q^k}$ denote the degree $k$ finite extension of $\F_{q}$ and $\mathbb{F}_{q^k}^*$ denote the set of non-zero elements in $\F_{q}$. Assume that $\zeta_p$ is a fixed primitive $p$-th root of unity in $\C$. For any Laurent polynomial $f(x_1,\ldots, x_n)\in \F_{q}[x_1^{\pm1},\ldots, x_n^{\pm1}]$, the $k$-th exponential sum associated to $f$ is defined by,
$$S^*_k(f)=\sum_{x_i \in \F^*_{q^k}} \zeta^{ \mathrm{Tr_k}  f(x_1, \ldots, x_{n})}_p, \ \ \ k=1,2,3,\ldots$$
where $\mathrm{Tr_k}: \F_{q^k} \mapsto\F_{p}$ is the trace map. In analytic number theory, it's a classical problem to give good estimates for archimedean and non-archimedean sizes of $S^*_k(f)(1\leq k < \infty)$. To understand the sequence of exponential sums, we usually evaluate the reciprocal roots or poles of the associated generating L-functions given by,
$$L^*(f,T)=\exp  \left(\sum^\infty _ {k=1} S^*_k (f) \frac{T^k}{k}\right) \in \mathbb{Q}(\zeta_p)[[T]].$$ Deligne's theorem on the Riemann hypothesis provides general information about the complex valuations of the zeros and poles. Also it's well known that the roots and poles of $L^*(f,T)$ are $\ell$-adic units if $\ell$ is a prime distinct from $p$. So the remaining interesting part is the $p$-adic valuation of $L^*(f,T).$ 

In this paper, we study $p$-adic estimate for the following family of Laurent polynomials over $\mathbb{F}_q$, denoted by $\mathcal{F}$:
$$f(x_1,x_2\ldots ,x_{n+1})=\sum_{i=1}^na_i x_{n+1}\left(x_i+\frac{1}{x_i}\right)+a_{n+1} x_{n+1}+\frac{1}{x_{n+1}}$$
where $a_i \in \F^*_{q}$, $i=1,2, \ldots, n+1$. Evaluating the generating exponential sum associated to $f \in \mathcal{F}$ is vital in analytic number theory. For instance, in Iwaniec's work \cite{Iwaniec1990} on small eigenvalues of the Laplace-Beltrami operator acting on automorphic functions with respect to $\Gamma_0(p)$, he improved the lower bound of eigenvalues conjectured by Selberg via the  estimate for $S_1^*(f)$ ($f \in \mathcal{F}$). 

For non-degenerate Laurent polynomial $f \in \mathcal{F}$ with $n+1$ variables, the associated L-function $L^*(f,T)^{(-1)^n}$ is a polynomial of degree $2^{n+1}$ by Adolphson-Sperber's theorem \cite{AS1989}. To understand the $p$-adic arithmetic property of $L^*(f,T)$, we compute the corresponding Newton polygon of $L^*(f,T)^{(-1)^n}$, denoted by NP($f$). The computation of NP($f$) turns out to be extremely complicated in general, although a standard lower bound of the Newton polygon is known. Adolphson and Sperber\cite{AS1989} constructed a certain combinatorial lower bound called the \emph{Hodge polygon}, the vertices of which are expressed in terms of Hodge numbers. Compared to NP($f$), the corresponding Hodge polygon is much easier to compute. So the problem is reduced to determining the shape of Hodge polygon and when the Newton polygon coincides with the corresponding Hodge polygon. 

Let $\vec{a}=(a_1,\ldots,$$a_{n+1})$ denote the coefficients of a non-degenerate Laurent polynomial $f \in \mathcal{F}$ and $\Delta_{n}=\Delta(f)$ denote the Newton polyhedron of $f$. Based on Wan's decomposition theorem\cite{Dwan2004}, Zhang and Feng \cite{Zhang2013OnLO} computed the Hodge polygon HP($\Delta_n$) and proved that this family $\mathcal{F}$ is generically ordinary for any prime $p$, i.e., there exsits a non-zero polynomial $h_p(\Delta_{n})(\vec{a})\in \mathbb{F}_p[a_1,\ldots,a_{n+1}]$ satisfying: if $h_p(\Delta_{n})(\vec{a})\neq0$, the Newton polygon of $L^*(f,T)^{(-1)^n}$ coincides with its lower bound HP($\Delta_n$). In Dwork's terminology, $h_p(\Delta_{n})(\vec{a})$ is called a \emph{Hasse polynomial} which is the main object to study in this paper.

Wan provided a general method to directly calculate the Hasse polynomials\cite{Dwan2004} but the method becomes insufficient for higher dimensional Newton polyhedrons. Based on Wan's method, Zhang and Feng\cite{Zhang2013OnLO} obtained an explicit formula of Hasse polynomials in low dimensions, i.e., $n \leq 3$. Let $h_p(\Delta_{n}, \leq1)$ denote the factor of $h_p(\Delta_{n})$ satisfying: if $h_p(\Delta_{n},\leq1)\neq 0$, Newton polygon coincides with HP($\Delta_n$) for all sides of slope $\leq 1$. We give an explicit formula for $h_p(\Delta_{n}, \leq1)$ where $n \in \mathbb{Z}_{\geq 0}$. 

\begin{Th} Let $f(x_1,\ldots,x_{n+1})\in \mathcal{F}$ be a non-degenerate laurent polynomial with $\Delta_n=\Delta(f)$. When $n\geq2$, a Hasse polynomial of slope at most one side can be taken to be,
\begin{equation*}
h_p(\Delta_n,\leq1)(\vec{a}) =\sum_{\substack{0 \leq v_1+\ldots +v_n \leq \frac{p-1}{2} \\ v_1, \ldots v_n \in \Z_{\geq0}}} \frac{a_1^{2v_1}a_2^{2v_2}\ldots a_n^{2v_n}a_{n+1}^{p-1-2(\sum_{i=1}^n v_i)}}{({v_1}!{v_2}!\ldots {v_n}!)^2(p-1-2(\sum_{i=1}^n v_i))!}.
\end{equation*}
\end{Th}

Zhang-Feng's formula for $n=3$ case is very complicated that involves the determinant of a $4\times 4$ matrix whose entries are all polynomials. In this paper, we provide a much simpler formula for the $n=3$ case based on Denef-Loeser's theorem\cite{Denef1991} and the symmetric property of NP($f$), which answers an open question of Zhang and Feng\cite{Zhang2013OnLO}. 
\begin{Th} For $n=3$, let  $f(x_1, \ldots, x_4)\in \mathcal{F}$ be a non-degenerate Laurent polynomial with $\Delta(f)=\Delta_3$. A Hasse polynomial of $\Delta_3$ can be written as,
$$h_p(\Delta_3)(\vec{a})=h_p (\Delta_3,\leq1) (\vec{a})=\sum_{\substack{0 \leq v_1+v_2+v_3 \leq \frac{p-1}{2} \\ v_1, v_2, v_3 \in \Z_{\geq0}}} \frac{a_1^{2v_1}a_2^{2v_2}a_3^{2v_3}a_{4}^{p-1-2(\sum_{i=1}^3 v_i)}}{({v_1}!{v_2}!{v_3}!)^2(p-1-2(\sum_{i=1}^3 v_i))!}.$$ 

In particular, NP($f$)=HP($\Delta_3$) if and only if $h_p(\Delta_3)(\vec{a})\not\equiv 0(\mathrm{mod}\ p)$ where $\vec{a}$ is the vector of coefficients of $f$.
\end{Th}

Furthermore, we are interested in the irreducibility of the Hasse polynomials. For $n \geq 3$ and $p\leq7$, we proved that the Hasse polynomial $h_p(\Delta_n,\leq1)(\vec{a})$ is irreducible over $\overline{\mathbb{F}}_p$. Based on this fact, we have the following hypothesis of the irreducibility. 

\begin{Conj} Let $p$ be an odd prime and $n\geq 3$. Then the Hasse polynomial $h_p(\Delta_n, \leq1)(\vec{a})$ is irreducible over $\overline{\F}_p$. 
\end{Conj}

This paper is organized as follows. In section 2, we review some basic definitions, technical theorems and relevant results including Dwork's trace formula, Wan's decomposition theory and Zhang-Feng's results. In section 3, we rigorously prove the two theorems concerning explicit formulas of the Hasse polynomials. In section 4, we study the irreducibility of the Hasse polynomials and provide some open problems.

\section{Preliminary}

\subsection{Rationality of the generating L-function.} 
Let $p$ be a prime and $q=p^a$ for some positive integer $a$. Let $\F_{q}$ be the finite field of $q$ elements with characteristic $p$ and $\F_{q^k}$ denote the degree $k$ finite extension of $\F_{q}$. For any Laurent polynomial $f(x_1,\ldots, x_n)\in \F_{q}[x_1^{\pm1},\ldots, x_n^{\pm1}]$, we define the associated $k$-th exponential sum as follows,
$$S^*_k(f)=\sum_{x_i \in \F^*_{q^k}} \zeta^{ \mathrm{Tr_k}  f(x_1, \ldots, x_{n})}_p, \ \ \ k=1,2,3,\ldots$$
where $\zeta_p$ is a fixed primitive $p$-th root of unity in $\C$ and $\mathrm{Tr_k}$ denotes the trace map from $\F_{q^k}$ to $\F_{p}$. 

In analytic number theory, it's a classical problem to give a good estimate for $|S^*_k(f)|$. In order to obtain the absolute values of the exponential sums, we usually study the generating L-functions. For a Laurent polynomial $f\in  \F_{q}[x_1^{\pm1},\ldots, x_n^{\pm1}]$, the generating L-function is defined to be,
$$L^*(f,T)=\exp  \left(\sum^\infty _ {k=1} S^*_k (f) \frac{T^k}{k}\right) \in \mathbb{Q}(\zeta_p)[[T]].$$ By a theorem of Dwork-Bombieri-Grothendieck\cite{Dwork1962OnTZ, Grothendieck1964}, the generating L-function is a rational function,
$$L^*(f,T)=\frac{\prod^{d_1}_{i=1}(1-\alpha_iT)}{\prod^{d_2}_{j=1}(1-\beta_jT)},$$ where all the reciprocal roots and zeros are non-zero algebraic integers. After taking logarithmic derivatives, we have the formula,
$$S_k^*(f)=\sum^{d_2}_{j=1}\beta_j^k-\sum^{d_1}_{i=1}\alpha_i^k, \  \  \ k=1,2,3,\ldots$$ 
The formula implies that the zeros and poles of the generating L-function contain critical information of the exponential sums.

From Deligne's theorem on Riemann hypothesis \cite{Deligne1980}, the complex absolute values of reciprocal zeros and poles are bounded as follows,
$$|\alpha_i|=q^{u_i/2}, |\beta_j|=q^{v_j/2}, u_i\in \mathbb{Z} \cap [0,2n], v_j\in \mathbb{Z} \cap [0,2n].$$ Restricting to non-degenerate Laurent polynomials, Adolphson and Sperber \cite{AS1989} showed that the associated L-function $L^*(f,T)^{(-1)^{n-1}}$ is a polynomial of degree $n!\mathrm{Vol(\Delta)}$. Here $\mathrm{Vol(\Delta)}$ is the volume of the Newton polyhedron of $f$ and the precise definition of Newton polyhedron will be given in the next subsection. They also got a sharper estimate for the archimedean size of the reciprocal zeros, i.e.,  $$L^*(f,T)^{(-1)^{n-1}}=\sum^{n! \mathrm{Vol(\Delta)}}_{i=1} A_i(f)T^i=\prod^{n! \mathrm{Vol(\Delta)}}_{i=1}(1-\alpha_i T),$$
where $A_i(f) \in \Z[\zeta_p]$ and $|\alpha_i| \leq q^{n/2}$. For non-archimedean values, Deligne\cite{Deligne1980} proved that $|\alpha_i|_\ell=|\beta_j|_\ell=1$ when $\ell$ is a prime and $\ell\neq p$. So the remaining interesting part is the $p$-adic absolute value which will be studied in this paper.

For non-degenerate Laurent polynomials, Adolphson and Sperber proved that the associated L-function $L^*(f,T)^{(-1)^{n-1}}$ is a polynomial over $\mathbb{Z}[\zeta_p]$. If we want to calculate $p$-adic absolute value of the reciprocal roots of a polynomial, it's equivalent to determine its Newton polygon. Adolphson and Sperber\cite{AS1989} proved that the Newton polygon of the associated L-function of a non-degenerate Laurent polynomial lies above a certain topological or combinatorial lower bound called \emph{Hodge polygon} and denoted by HP($\Delta$) which depends only on the Newton polyhedron $\Delta$ and is easier to compute. We reduce the problem to determining the shape of Hodge polygon and when the Newton polygon coincides with the corresponding Hodge polygon.

\subsection{Newton polygon and Hodge polygon} 
Let $f(x_1, \ldots x_n)=\sum^J_{j=1} a_j x^{V_j}$ be a Laurent polynomial with $a_j \in \F ^*_{q}$ and $V_j=(v_{1j},\ldots, v_{nj})\in \Z^n$. The \emph{Newton polyhedron} of $f$, denoted by $\Delta(f)$, is defined to be the convex closure in $\R^n$ generated by the origin and the lattice points $V_j$ ($1\leq j \leq J$). If $\delta$ denotes a subset of $\Delta(f)$, then the restriction of $f$ to $\delta$ is defined to be $f^{\delta}=\sum_{V_j\in \delta}a_j x^{V_j}$. Generally, we require Laurent polynomials to be \emph{non-degenerate} defined as follows. 
 
\begin{Def}A Laurent polynomial $f$ is called non-degenerate if for each closed face $\delta$ of $\Delta(f)$ of arbitrary dimension which doesn't contain the origin, the $n$ partial derivatives, $\{ \frac{\partial f^{\delta}}{\partial x_1 }, \ldots, \frac{\partial f^{\delta}}{\partial x_n} \}$, have no common zeros with $x_1\ldots x_n \neq 0$ over $\bar{\F}_q$.
 \end{Def}

For any non-degenerate Laurent polynomial $f$ in $n$ variables, Adolphson and Sperber's theorem \cite{AS1989} shows that the associated L-function is of the following form, 
$$L^*(f,T)^{(-1)^{n-1}}=\sum^{n! \mathrm{Vol(\Delta)}}_{i=1} A_i(f)T^i=\prod^{n! \mathrm{Vol(\Delta)}}_{i=1}(1-\alpha_i T)$$ 
where $A_i(f) \in \Z[\zeta_p]$ and $|\alpha_i| \leq q^{n/2}$. Deligne's integrality theorem implies that the $p$-adic absolute values of reciprocal roots are given by $|\alpha_i|_p=q^{-r_i}$ where $r_i\in \mathbb{Q}\cap[0,n]$. For simplicity, we normalize $p$-adic absolute value to be $|q|_p=q^{-1}$. We can use $q$-adic Newton polygon to get more information of the $q$-adic absolute values of zeros since the shape of Newton polygon will provide the valuation information of all the reciprocal roots of the polynomial. Define the $q$-adic Newton polygon as follows.

\begin{Def}\label{m1}
Let $g(T)=\sum^n_{i=0}b_i T^i$ $\in 1+T\overline{\Q}_p[T]$ where $\overline{\Q}_p$ is the algebraic closure of $\Q_p$. The $q$-adic Newton polygon of $g(T)$ is defined to be the lower convex closure of the set of points $\{(k,\ord_q(b_k)) | k=0, 1,\ldots, n \}$ in $\R^2$.
\end{Def}

Here ord$_q$ denotes the standard $q$-adic ordinal on $\bar{\mathbb{Q}}_p$ where the valuation is normalized to be $\text{ord}_q(q)=1$. The following lemma \cite{koblitz2012p} relates the $q$-adic Newton polygon to $q$-adic valuation of reciprocal roots.

\begin{Lemma}In the above notation, let $g(T)=(1-{\alpha_1}T)\ldots(1-{\alpha_n}T)$ be the factorization of $g(T)$ in terms of reciprocal roots $\alpha_i \in \bar{\Q}_p$. Let $\lambda_i=\text{ord}_q\alpha_i$. If $\lambda$ is a slope of the $q$-adic Newton polygon with horizontal length $l$, then precisely $l$ of the $\lambda_i$ are equal to $\lambda$.
\end{Lemma}

For a given Laurent polynomial $f$ in $n$ variables, let NP($f$) denote the $q$-adic Newton polygon of $L^*(f,T)^{(-1)^{n-1}}$ and let $\Delta$ denote the Newton polyhedron $\Delta(f)\in \mathbb{R}^n$. Define $C(\Delta)$ to be the cone generated by $\Delta$ in $\R^{n}$. For any point $u\in \R^{n}$, the weight function $w(u)$ is defined to be the smallest non-negative real number $c$ such that $u\in c\Delta$ and $w(u)=\infty$ if such $c$ doesn't exist. Generally we have $w(\mathbb{Z}^n)\subseteq \frac{1}{D(\Delta)}\mathbb{Z}_{\geq0}\cup \{ + {\infty} \}$, where $D=D(\Delta)$ is the least common multiple of denominators of weights of lattice points on all the co-dimentional 1 faces of $\Delta$. Let $$W_{\Delta}(k)=\# \left \{u \in \Z^n | w(u)= \frac{k}{D}\right \}$$ be the number of lattice points in $\Z^n$ with weight $k/D$. Then define the Hodge polygon of a given polyhedron $\Delta$ as follows.

\begin{Def}The Hodge polygon HP($\Delta$) of $\Delta$ is the lower convex polygon in $\R^2$ with vertices (0,0) and 
$$P_k=\left( \sum^k_{m=0}H_{\Delta}(m), \frac{1}{D}\sum^k_{m=0}m H_{\Delta}(m)\right), \  k=0,1,\ldots, nD$$ 
where $H_{\Delta}(k)=\sum^{n}_{i=0}(-1)^i \binom{n}{i}W_{\Delta}(k-iD)$, $k=0,1,\ldots, nD.$\\ 
\indent \ \ That is, HP($\Delta$) is a polygon starting from origin (0,0) with a slope $k/D$ side of horizontal length $H_{\Delta}(k)$ for $k=0,1,\ldots, nD$. The vertex $P_k$ is called a break point if $H_{\Delta}(k+1)\neq 0$ where $k=1,2,\ldots,nD-1$. 
\end{Def}

Here the horizontal length $H_{\Delta}(k)$ represents the number of  lattice points of weight $k/D$ in a certain fundamental domain corresponding to a basis of the $p$-adic cohomology space used to compute the L-function. Adolphson and Sperber\cite{AS1989} proved that $H_{\Delta}(k)$ coincides with the usual Hodge number in the toric hypersurface case in which case $D=1$. Thus this lower bound is called the Hodge polygon.

\subsection{Hasse domain and Hasse polynomial}
For a fixed $n$-dimensional integral polytope $\Delta \subset \mathbb{R}^n$ containing origin, let $\mathcal{N}_p(\Delta)$ be the parameter space of Laurent polynomials $g$ defined over $\overline{\F}_p$ with $\Delta(g)=\Delta$. Then $\mathcal{N}_p(\Delta)$ is a smooth affine variety over $\mathbb{F}_p$. Let $\mathcal{M}_p(\Delta)$ be the subset of $\mathcal{N}_p(\Delta)$ consisting of all non-degenerate Laurent polynomials. By the definition of non-degeneracy, $\mathcal{M}_p(\Delta)$ is a Zariski open smooth affine subset of $\mathcal{N}_p(\Delta)$ which is the complement of a certain discriminant locus. Also $\mathcal{M}_p(\Delta)$ is non-empty for $p$ sufficiently large, i,e, $p > n!\text{Vol}(\Delta)$. Then it's natural to consider how Newton polygon NP($f$) varies as $f$ varies in $\mathcal{M}_p(\Delta)$. Grothendieck specialization theorem\cite{wan2000higher} implies there is a lower bound for the Newton polygon of $L^*(f,T)^{(-1)^{n-1}}$ for $f\in \mathcal{M}_p(\Delta)$ and the lower bound is attained for polynomials in some Zariski open dense subset of $\mathcal{M}_p(\Delta)$. Define the generic Newton polygon to be
$$\text{GNP}(\Delta,p):=\inf_{f\in \mathcal{M}_p}\text{NP}(f).$$ Based on Adolphson and Sperber's theorem\cite{AS1989}, we have the following inequalities. 

\begin{Prop}
For every prime p and $f \in \mathcal{M}_p(\Delta)$, we have $$\text{NP}(f) \geq \text{GNP}(\Delta,p) \geq \text{HP}(\Delta).$$ 
Furthermore, the endpoints of NP($f$) and NP($\Delta$) coincide.
\end{Prop}

\begin{Def} A Laurent polynomial $f$ is called ordinary if NP($f$) = HP($\Delta$). The family $\mathcal{M}_p(\Delta)$ is called generically ordinary if GNP($\Delta,p$)=HP($\Delta$).
\end{Def}
 
Let $\mathcal{H}_p(\Delta)=\{g\in \mathcal{M}_p(\Delta) | \text{NP} (g) = \text{HP}(\Delta)\}$ be the subset of $\mathcal{M}_p(\Delta)$ containing non-degenerate ordinary Laurent polynomials, which is called the $\textit{Hasse}$ $\textit{domain}$ in Dwork's terminology. Similarly, let $\mathcal{H}_p(\Delta, \leq k)=\{g\in \mathcal{M}_p(\Delta) |$ $\text{NP} (g) $ = $ \text{HP}(\Delta)$ for all sides of slopes $\leq k/D\}$ and $\mathcal{H}_p(\Delta, k) =\{g\in \mathcal{M}_p(\Delta) | \text{NP} (g) = \text{HP}(\Delta)$ for the slope $k/D \ \text{side} \}$. It's easy to check that $\mathcal{H}_p(\Delta)$, $\mathcal{H}_p(\Delta, \leq k)$ and $\mathcal{H}_p(\Delta, k)$ are Zariski-open subsets of $\mathcal{M}_p(\Delta)$ (possibly empty). A basic question is whether $\mathcal{H}_p(\Delta)$ is empty or not. If $\mathcal{H}_p(\Delta)=\varnothing$, then GNP($\Delta,p$) > HP($\Delta$) which implies $\mathcal{M}_p(\Delta)$ is not generically ordinary. In this case, every Laurent polynomial in $\mathcal{M}_p(\Delta)$ is not ordinary.

Let $\vec{a}$ denote the coefficients of $f \in \mathcal{M}_p(\Delta)$. If $\mathcal{H}_p(\Delta)$ is not empty, there exists a non-zero polynomial $h_p(\Delta)(\vec{a})$ $\in \mathbb{F}_p[\vec{a}]$ such that if the coefficients of $f \in \mathcal{M}_p(\Delta)$ satisfying $h_p(\Delta)(\vec{a})\neq 0$, then NP($f$) = HP($\Delta$). So we have the following proposition.

\begin{Prop} 
If $\mathcal{H}_p(\Delta)$ is not empty, then $\mathcal{H}_p(\Delta)$ is Zariski dense in $\mathcal{M}_p(\Delta)$ and the complement of $\mathcal{H}_p(\Delta)$ is a hypersurface determined by a non-zero polynomial $h_p(\Delta)$ over $\mathbb{F}_p$ which is called a Hasse polynomial with respect to $\Delta$.
\end{Prop}

Similarly, let $h_p(\Delta, k)$ be a Hasse polynomial of slope $k/D$ side and let $h_p(\Delta, \leq k)$ be a Hasse polynomial of all sides with slope $\leq k/D$. For a non-degenerate Laurent polynomial $f$ with $\Delta(f)=\Delta$, Hasse polynomials determine the amount of coincidence between NP($f$) and HP($\Delta$).

\subsection{Dwork's trace formula} 
Let $p$ be a prime and $q=p^a$ for some positive integer $a$. Let $\Q_p$ be the field of $p$-adic numbers and $\Omega$ be the completion of $\overline{\Q}_p$. Pick a fixed primitive $p$-th root of unity in $\Omega$, denoted by $\zeta_p$. In $\Q_p(\zeta_p)$, choose a fixed element $\pi$ satisfying
$$\sum^{\infty}_{m=0}\frac{\pi^{p^m}}{p^m}=0 \ \ \text{and} \ \ \text{ord}_p \pi =\frac{1}{p-1}.$$ By Krasner's lemma, it's easy to check $Q_p(\pi)=Q_p(\zeta_p)$. Let $K$ be an unramified extension of $\mathbb{Q}_p$ of degree $a$ and $\Omega_a$ be the compositum of $Q_p(\zeta_p)$ and $K$. 
\begin{center}
\begin{tikzpicture}[node distance = 1.5cm, auto]
      \node (Q) {$\mathbb{Q}_p$};
      \node (K) [above of=Q, left of=Q] {$K$};
      \node (Z) [above of=Q, right of=Q] {$Q_p(\pi)$};
      \node (O) [above of=Q, node distance = 3cm] {$\Omega_a$};
      \draw[-] (Q) to node {$a$} (K);
      \draw[-] (Q) to node [swap] {$p-1$} (Z);
      \draw[-] (K) to node {$$} (O);
      \draw[-] (Z) to node [swap] {$$} (O);
      \end{tikzpicture}
\end{center}
Lift the Frobenius automorphism $x\mapsto x^p$ of Gal($\mathbb{F}_q/\mathbb{F}_p$) to a generator $\tau$ of Gal($K/\mathbb{Q}_p$) and extend it to $\Omega_a$ by requiring $\tau(\pi)=\pi$. For the primitive $(q-1)$-th root of unity $\zeta_{q-1}$, we have $\tau(\zeta_{q-1})=\zeta_{q-1}^p$.

Let $E_p(t)$ be the Artin-Hasse exponential series:
$$E_p(t)=\exp\left(\sum^{\infty}_{m=0}\frac{t^{p^m}}{p^m}\right)=\sum_{m=0}^{\infty}\lambda_m t^m.$$ Based on Dwork's lemma, it's easy to check that the coefficients of $E_p(t)$ are $p$-adic integers, i.e., $E_p(t)\in \mathbb{Z}_p[[x]]$. After simple calculation, we have
$$\lambda_m=
\begin{cases}\frac{1}{m!},& \text{if} \ 0\leq m \leq p-1,\\
\frac{1}{m!}+ \frac{1}{p(m-p)!} ,& \text{if} \ p\leq m \leq 2p-1.
\end{cases}$$
In Dwork's terminology, a splitting function $\theta(t)$ is defined to be,
$$\theta(t)=E(\pi t)=\sum_{m=0}^{\infty}\lambda_m\pi^mt^m.$$
Note that $\theta(1)$ is a primitive $p$-th root of unity that can be identified with $\zeta_p$ in $\Omega$.

\indent \ \  Consider a Laurent polynomial $f\in $ $\F_q[x_1^{\pm1},x_2^{\pm1}, \ldots, x_n^{\pm1}]$ given by
 $$f=\sum_{j=1}^J \bar{a}_j x^{V_j}$$ 
where $V_j \in {\Z}^n$ and $\bar{a}_j \in \F_q^{*}$. Let $a_j$ be the Teichm\"{u}ller lifting of $\bar{a}_j $ in $\Omega$ satisfying $a_j^q=a_j$. Let
$$F(f,x)=\prod_{j=1}^J\theta(a_j x^{V_j})=\sum_{r \in {\Z}^n} F_r(f)x^r$$
 with coefficients given by 
 \begin{equation}\label{eq3}
 F_r(f)=\sum_u (\prod^{J}_{j=1} \lambda_{u_j} a_j^{u_j}) \pi^{u_1+\ldots+u_{J}}
 \end{equation}
 where the sum is over all the solutions of the following linear system
$$\sum^{J}_{j=1}u_jV_j=r, \  u_j \in \Z_{\geq 0} $$ and $\lambda_m$ is $m$-th coefficient of the Artin-Hasse exponential series $E_p(t)$.

Let $\Delta$ be the Newton polyhedron of $f$ and $L(\Delta)=\mathbb{Z}^{n}\cap C(\Delta)$ be the set of lattice points in the closed cone generated by origin and $\Delta$. Recall that for a given point $r\in \mathbb{R}^n$, the weight function is defined to be $$w(r): =\inf_{\vec{u}}\left \{ \sum_{j=1}^J u_j |\sum_{j=1}^J u_jV_j=r,\ u_j\in \R_{\geq 0}\right\}.$$ Then let $A_1(f)$ be an infinite matrix whose rows and columns are indexed by the lattice points in $L(\Delta)$ with respect to the weights: 
\begin{equation}\label{eq4}
A_1(f)=(a_{r,s}(f))=(F_{ps-r}(f)\pi^{w(r)-w(s)})
\end{equation}
where $r, s\in L(\Delta)$. In Dwork's terminology, $A_1(f)$ is called the infinite semilinear Frobenius matrix.

By Dwork's trace formula, we can identify the associated L-function with a product of Fredholm determinants,
$$L^{*}(f,T)^{(-1)^{n-1}}=\prod_{i=0}^{n}\text{det}\left(I-Tq^{i}A_a(f)\right)^{(-1)^i\binom{n}{i}}$$ 
where $A_a(f)=A_1A_1^{\tau}\cdots A_1^{\tau^{a-1}}$ is the infinite linear Frobenius matrix.

Based on the fact that $\ord_p F_r(f)\geq \frac{w(r)}{p-1}$, we have the estimate
$$\ord_p( a_{r,s}(f)) \geq \frac{w(ps-r)+w(r)-w(s)}{p-1}\geq w(s).$$ Let $\xi$ be an element in $\Omega$ satisfying $\xi^D=\pi^{p-1}$. Then $A_1(f)$ can be written in the following block form,
\begin{equation*}
A_1(f)
=\begin{pmatrix}
A_{00} &  \xi A_{01}  & \cdots\ & {\xi}^iA_{0i}&\cdots\\
A_{10} &  \xi A_{11}  & \cdots\ & {\xi}^iA_{1i}&\cdots\\
 \vdots & \vdots & \ddots  & \vdots  \\
 A_{i0} &  \xi A_{i1}  & \cdots\ & {\xi}^iA_{ii}&\cdots\\
 \vdots & \vdots & \ddots  & \vdots
\end{pmatrix}
\end{equation*}
where the block $A_{ii}$ is a $p$-adic integral $W_{\Delta}(i) \times W_{\Delta}(i)$ matrix and $W_{\Delta}(i)=\#\{u \in \Z^n | w(u)= \frac{i}{D}\}$.

\begin{Def} Let $P(\Delta)$ be a polygon in $\mathbb{R}^2$ with vertices (0,0) and 
$$\left(\sum^k_{i=0}W_{\Delta}(i), \frac{1}{D}\sum^k_{i=0} i W_{\Delta}(i)\right),  \ k=0,1,2,\ldots$$ 
\end{Def}
Generally, $P(\Delta)$ is identified with the chain level version of Hodge bound which is the lower bound of the Newton polygon. Based on the block form of $A_1(f)$, we have the following result.
\begin{Prop}\label{prop29} Let $f$ be a Laurent polynomial with $\Delta(f)=\Delta$, then
\begin{enumerate}
\item The $p$-adic Newton polygon of $\det(I-TA_1(f))$ lies above  $P(\Delta$).
\item The $q$-adic Newton polygon of $\det(I-TA_a(f))$ lies above $P(\Delta)$.
\end{enumerate}
\end{Prop}

\subsection{General method for computing Hasse polynomials}

In this subsection, we use the same notations as the previous subsection and provide a standard way to compute a Hasse polynomial which represents the ordinary property of a given non-degenerate Laurent polynomial. The following theorem relates the ordinary property of $f$ to the $p$-adic Newton polygon of $\det(I-TA_1(f))$.  
\begin{Th}[\cite{Dwan2004}]Let $f$ be a Laurent polynomial of $n$ variables with $\Delta=\Delta(f)$. Assume that the L-function $L^*(f,T)^{(-1)^{n-1}}$ is a polynomial. Then 
\begin{enumerate}
\item NP($f$) = HP($\Delta$) if and only if the $q$-adic Newton polygon of $\det(I-TA_a(f))$ coincides with its lower bound P($\Delta$) if and only if the $p$-adic Newton polygon of $\det(I-TA_1(f))$ coincides with its lower bound P($\Delta$).
\item NP($f$) coincides with HP($\Delta$) for all sides with slopes $\leq k/D$ if and only if the $p$-adic Newton polygon of $\det(I-TA_1(f))$ computed with respect to $p$ coincides with P($\Delta$) for the sides with slopes $\leq k/D$.
\end{enumerate}
\end{Th}

Recall that the Hasse domain $\mathcal{H}_p(\Delta)=\{g\in \mathcal{M}_p(\Delta) | \text{NP} (g) = \text{HP}(\Delta)\}$ is a Zariski open subset of $\mathcal{M}_p(\Delta)$. If $\mathcal{H}_p(\Delta)$ is not empty, its complement is a hypersurface determined by Hasse polynomial $h_p(\Delta)$. Similarly, $h_p(\Delta, k)$ defines the complement of $\mathcal{H}_p(\Delta, k)=\{g\in \mathcal{M}_p(\Delta) | \text{NP} (g) = \text{HP}(\Delta)$ for slope-$k/D$ side$\}$. Let det($I-A_1(f)$) = $\sum_{j=0}^{\infty}c_jT^j$. Newton polygon of det($I-A_1(f)$) computed with respect to $p$ is the lower convex closure of \{$(j,\ord_p(c_j))$ $|$ $k=0,1,\ldots \} $. Let $(j, P(\Delta, j))$ be a point on $P(\Delta)$ for $j\in \mathbb{Z}_{\geq 0}$. By Proposition \ref{prop29}, we have $\mathrm{ord}_p(c_j) \geq P(\Delta, j)$. Let $j'=\sum^k_{i=0}W_{\Delta}(i)$. From the block form of $A_1(f)$, it's easy to check that 
$$c_{j'}=\xi^{\sum^k_{i=0}i W_{\Delta}(i)}\text{det}
\begin{pmatrix}
A_{00} &  A_{01}  & \cdots\ & A_{0k}\\
A_{10} &  A_{11}  & \cdots\ & A_{1k}\\
 \vdots & \vdots & \ddots  & \vdots  \\
 A_{k0} &  A_{k1}  & \cdots\ & A_{kk}
\end{pmatrix}
+\xi^{1+\sum^k_{i=0}i W_{\Delta}(i)}\cdot u_{j'} 
$$ where $u_{j'}$ is a p-adic integer and $\xi^{D}=-p$. So $\text{ord}_p(c_{j'})\geq \frac{1}{D}\sum^k_{i=0}i W_{\Delta}(i)=P(\Delta,j')$ and $\text{ord}_p(c_{j'}) = P(\Delta,j')$ if and only if
\begin{equation}\label{eq1}
h_p(\Delta, k): =\mathrm{det}\begin{pmatrix}
A_{00} &  A_{01}  & \cdots\ & A_{0k}\\
A_{10} &  A_{11}  & \cdots\ & A_{1k}\\
 \vdots & \vdots & \ddots  & \vdots  \\
 A_{k0} &  A_{k1}  & \cdots\ & A_{kk}
\end{pmatrix} \not\equiv 0 \mod p.
\end{equation}
Consequently, we have the following formula for computing the Hasse polynomial,\\
\begin{equation}
h_p(\Delta)=\prod^{nD}_{k=0}h_p(\Delta, k).
\end{equation} 

\subsection{Wan's facial decomposition theorem and Zhang-Feng's result}
Wan develops some decomposition theorems to simplify the computation by decomposing the polyhedron $\Delta$ to small pieces \cite{Dwan2004}. We will use the following facial decomposition to compute the Hasse polynomial.

\begin{Th}[Facial Decomposition Theorem\cite{Dwan1993}]\label{th211}
Let $f$ be a non-degenerate Laurent polynomial over $\mathbb{F}_q$. Assume $\Delta(f)=\Delta$ is $n$-dimensional and $\delta_1,\ldots, \delta_h$ are all the co-dimension 1 faces of $\Delta$ which don't contain the origin. Let $f^{\delta_i}$ denote the restriction of $f$ to $\delta_i$. Then $f$ is ordinary if and only if $f^{\delta_i}$ is ordinary for $1\leq i\leq h$. 

Moreover for positive integer $k$, the Newton polygon of $\det(I-TA_1(f))$ coincides with its lower bound at the $k$-th vertex if and only if the Newton polygon of det$(I-TA_1(\delta_i, f^{\delta_i}))$ coincides with its lower bound at the $k$-th vertex for $i=1,\ldots,h$.
\end{Th}

Consider the following family of Laurent polynomials over finite field $\F_{q}$:
$$f(x_1,\ldots ,x_{n+1})=\sum_{i=1}^na_i x_{n+1}\left(x_i+\frac{1}{x_i}\right)+a_{n+1} x_{n+1}+\frac{1}{x_{n+1}}$$ 
where $a_i \in \F^*_{q}$, $i=1,2, \ldots, n+1$. Then $\Delta(f)$ is a polyhedron in $\mathbb{R}^{n+1}$ and let $\Delta_n = \Delta(f)$. Based on Wan's decomposition theorem \cite{Dwan1993}, Zhang and Feng\cite{Zhang2013OnLO} gave the following results.

\begin{Th}[\cite{Zhang2013OnLO}]\label{th212} Let $f$ be a Laurent polynomial in that family with $\Delta_n=\Delta(f)$.
\begin{enumerate}
\item Then $f$ is non-degenerate if and only if $\pm2a_1\pm2a_2 \ldots \pm 2a_n+a_{n+1}\neq 0$, i.e.,\\
 $\mathcal{M}_p(\Delta_n)=\{a\in {\bar{\F}_p}^{n+1} | a_1 \cdots a_{n+1} \prod (\pm2a_1\pm2a_2 \ldots \pm 2a_n+a_{n+1})\neq 0\}.$
 \item If $f\in \mathcal{M}_p(\Delta_n)$, the associated L-function $L^{*}(f,T)^{(-1)^n} $ is a polynomial of degree $2^{n+1}$, i.e., $(n+1)!\mathrm{Vol}(\Delta_n)=2^{n+1}$.
 \item The Hodge polygon HP($\Delta_n$) is a lower convex polygon in $\R^2$ with vertices (0,0) and  
$$\left( \sum^k_{m=0}\binom{n+1}{m}, \sum^k_{m=0}m \binom{n+1}{m}\right), \  k=0,1,\ldots, n+1,$$ 
i.e., $D=1$ and  $H_{\Delta_n}(m)=\binom{n+1}{m}$ for $m=0,1,\ldots, n+1$.\\
\indent \ \ That is, the Hodge polygon HP($\Delta_n$) $\in \mathbb{R}^2$ consists of $n+2$ line segments starting from initial point (0,0) and has sides of increasing slopes, i.e., the $j$-th segment has slope $j-1$ with horizontal length $\binom{n+1}{j-1}$, $j=1,\ldots,n+2$.
\item This family is generically ordinary for any prime $p$, i.e., there exsits a non-zero polynomial $h_p(\Delta_{n})(\vec{a})\in \mathbb{F}_p[a_1,\ldots,a_{n+1}]$ satisfying: if $h_p(\Delta_{n})(\vec{a})\neq0$, the Newton polygon of $L^*(f,T)^{(-1)^n}$ coincides with its lower bound HP($\Delta_n$).
\end{enumerate} 
\end{Th}

\begin{Th}[\cite{Zhang2013OnLO}]\label{th213}Notations as above.
\begin{enumerate}
\item When $n=2$, a Hasse polynomial can be taken to be
\begin{equation*}
h_p(\Delta_2)(a_1, a_2, a_3) =\sum_{\substack{0 \leq v_1+v_2 \leq \frac{p-1}{2} \\ v_1, v_2 \in \Z_{\geq0}}} \frac{a_1^{2v_1}a_2^{2v_2} a_{3}^{p-1-2(v_1+v_2)}}{({v_1}!{v_2}!)^2(p-1-2(v_1+v_2))!}.
\end{equation*}
where $\vec{a}=(a_1,a_2, a_3)\in \mathcal{M}_p(\Delta_2)$. 
\item When $n=3$, a Hasse polynomial can be taken to be
\begin{equation*}
h_p(\Delta_3)(\vec{a}) =T(\vec{a})\sum_{\substack{0 \leq v_1+v_2+v_3 \leq \frac{p-1}{2} \\ v_1, v_2, v_3 \in \Z_{\geq0}}} \frac{a_1^{2v_1}a_2^{2v_2} a_3^{2v_3} a_{4}^{p-1-2(v_1+v_2+v_3)}}{({v_1}!{v_2}!{v_3}!)^2(p-1-2(v_1+v_2+v_3))!},
\end{equation*}
where $\vec{a}=(a_1,a_2, a_3, a_4)\in \mathcal{M}_p(\Delta_3)$ and $T(\vec{a})$ is the determinant of a $4\times 4$ matrix whose entries are all polynomials.
\end{enumerate}
\end{Th}

\section{Proof of Main Theorems}
In this section, we give the proof of our main theorem. 
\begin{Th} \label{m1} For $n\geq2$. Let $f$ be a non-degenerate laurent polynomial defined as $$f(x_1,\ldots ,x_{n+1})=\sum_{i=1}^na_i x_{n+1}\left(x_i+\frac{1}{x_i}\right)+a_{n+1} x_{n+1}+\frac{1}{x_{n+1}}$$ 
where $a_i \in \F^*_{q}$, $i=1,2, \ldots, n+1$. For $\Delta_n=\Delta(f)$, we have
\begin{equation}\label{eq31}
h_p(\Delta_n,\leq1)(\vec{a})=\sum_{\substack{0 \leq v_1+\ldots +v_n \leq \frac{p-1}{2} \\ v_1, \ldots v_n \in \Z_{\geq0}}} \frac{a_1^{2v_1}a_2^{2v_2}\ldots a_n^{2v_n}a_{n+1}^{p-1-2(\sum_{i=1}^n v_i)}}{({v_1}!{v_2}!\ldots {v_n}!)^2(p-1-2(v_1+v_2+\ldots +v_n))!}
\end{equation}
where $\vec{a}=(a_1,\ldots, a_{n+1})\in \mathcal{M}_p(\Delta_n).$
\end{Th}

\begin{proof} 
If we restrict $f$ to the co-dimensional 1 face $\delta_0: x_{n+1}=1$, we will get the following new Laurent polynomial 
$$g=f^{\delta_0}=f(x_1,\ldots ,x_{n+1})=\sum_{i=1}^na_i x_{n+1}\left(x_i+\frac{1}{x_i}\right)+a_{n+1} x_{n+1}$$ 
From Theorem \ref{th211} and Zhang-Feng's calculation \cite{Zhang2013OnLO}, we know that $f$ is ordinary if and only if $g$ is ordinary, i.e.,  NP($f$) coincides with its lower bound at the $k$-th vertex if and only if NP($g$) coincides with its lower bound at the $k$-th vertex. 

Let $\Delta_n=\Delta(f)$ and $\Delta'_n=\Delta(g)$. To obtain $h_p(\Delta_n,\leq1)$, it suffices to compute $h_p(\Delta'_n,\leq1)$. From formula (\ref{eq3}), (\ref{eq4}) and (\ref{eq1}), 
$$h_p(\Delta'_n,\leq1)=h_p(\Delta'_n,1)=
\mathrm{det}\begin{pmatrix}
A_{00} &  A_{01} \\
A_{10} &  A_{11}
\end{pmatrix}\mod{p.}$$
where $A_{00}=1$ and $A_{10}=(0,0,\ldots,0)^T$. So $h_p(\Delta'_n,\leq1)=\mathrm{det}A_{11}\mod p$. Based on Zhang and Feng's calculation\cite{Zhang2013OnLO}, $W_{\Delta'_n}(1)=2n+1$. So $A_{11}$ is a matrix of size $(2n+1) \times (2n+1)$. For simplicity, we enumerate the vertices of $\Delta'_n$. Let $V_1=(1,0,0,\ldots,0,1)$, $V_2=(-1,0,0,\ldots,0,1)$, $V_3=(0,1,0,\ldots,0,1)$, \dots, $V_{2n}=(0,0,\ldots,0,-1,1)$, $V_{2n+1}=(0,0,\ldots,0,1)$. If the vertices are suitably arranged, $A_{11}$ is a lower triangular matrix of the following form,
$$A_{11}=\bordermatrix{
               &V_{2n+1}&V_1&V_2&\cdots\ &V_{2n} \cr
V_{2n+1}& \large{*}   & 0    &0    &\cdots\ &0          \cr
V_{1}&  \large{*}  &\large{*}    &0    &\cdots\ &0          \cr
V_{2}&  \large{*}  &\large{*}    &\large{*}    &\cdots\ &0          \cr
 \vdots & \vdots & \vdots & \vdots & \ddots  & \vdots  \cr
V_{2n}&\large{*} & \large{*} & \large{*} &\cdots\ &\large{*}          \cr}=(b_{ij})_{1\leq i,j \leq 2n+1}.$$

So $h_p(\Delta'_n, \leq 1)=\mathrm{det}A_{11}\mod p=\prod_{k=1}^{2n+1}b_{kk}\mod p$.
By formula (\ref{eq3}) and (\ref{eq4}), 
\begin{equation} \nonumber
\begin{split}
b_{11}&=\frac{1}{p}F_{(0,\ldots,0,p-1)}(g)\\
&=\sum_{\substack{0 \leq v_1+\ldots +v_n \leq \frac{p-1}{2} \\ v_1, v_2, \ldots v_n \in \Z_{\geq0}}}\frac{1}{(v_1!)^2}\cdots\frac{1}{(v_n!)^2}\frac{1}{(p-1-2(\sum_{i=1}^n v_i))!}a_1^{2v_1}\ldots a_n^{2v_n}a_{n+1}^{p-1-2(\sum_{i=1}^n v_i)}
\end{split}
\end{equation}
and 
\begin{equation}
b_{ii}=
\begin{cases}
\frac{1}{(p-1)!}a_k^{p-1} &\text{if} \ i=2k, \\ 
\frac{1}{(p-1)!}a_k^{p-1} &\text{if} \ i=2k+1,
\end{cases}
\end{equation} where $1\leq k \leq n$. Since $a_k\in \F^*_q$ and $\mathrm{ord}_p(\frac{1}{(p-1)!})=0$, we can conclude $b_{ii} (i\geq2$) are trivial factors for the Hasse polynomial.
So we proved the following equation,
\begin{equation}\nonumber
\begin{split}
h_p(\Delta_n,\leq1)&=h_p(\Delta'_n,\leq1)\\
&=\sum_{\substack{0 \leq \sum_{i=1}^n v_i \leq \frac{p-1}{2} \\ v_1, v_2, \ldots v_n \in \Z_{\geq0}}}\lambda^2_{v_1}\lambda^2_{v_2}\ldots \lambda^2_{v_n}\lambda_{p-1-2(\sum_{i=1}^n v_i)}a_1^{2v_1}a_2^{2v_2}\ldots a_n^{2v_n}a_{n+1}^{p-1-2(\sum_{i=1}^n v_i)}\\
&=\sum_{\substack{0 \leq \sum_{i=1}^n v_i \leq \frac{p-1}{2} \\ v_1, v_2, \ldots v_n \in \Z_{\geq0}}}\frac{1}{(v_1!)^2}\ldots\frac{1}{(v_n!)^2}\frac{a_1^{2v_1}a_2^{2v_2}\ldots a_n^{2v_n}a_{n+1}^{p-1-2(\sum_{i=1}^n v_i)}}{(p-1-2(v_1+\ldots+v_n))!}
\end{split}
\end{equation}
\end{proof}

In the second part of this section, we will give a formula of the Hasse polynomial $h_p(\Delta_3)$ which is much simpler compared to Zhang and Feng's result \cite{Zhang2013OnLO}, stated in Theorem \ref{th213}. Before proving the formula,  we introduce a lemma which will be used in the proof. This lemma follows from Denef-Loeser's weight formula \cite{Denef1991}.
\begin{Lemma}\label{lm1}
Let $f\in  \F_{q}[x_1^{\pm1},\ldots, x_n^{\pm1}]$ be a non-degenerate Laurent polynomial. Assume that the Newton polyhedron of $f$ is an $n$-dimensional polytope in $\mathbb{R}^n$, denoted by $\Delta$. If the origin is an interior point of $\Delta$, then the associated L-function is purely of weight $n$, i.e.,  
$$L^*(f,T)^{(-1)^{n-1}}=\prod^{n!\mathrm{Vol}(\Delta)}_{i=1}(1-\alpha_i T)$$ where $|\alpha_i|=q^{n/2}$, $i=0,1,\ldots, n!\mathrm{Vol}(\Delta)$. 
\end{Lemma}
Then we have the following formula based on the symmetric property of the Newton polygons of L-functions derived from Lemma \ref{lm1}.
\begin{Th}
For $n=3$, let  $f(x_1, \ldots, x_4)\in \mathcal{F}$ be a non-degenerate Laurent polynomial with $\Delta_3=\Delta(f)$. A Hasse polynomial of $\Delta_3$ can be written as,
$$h_p(\Delta_3)(\vec{a})=h_p (\Delta_3,\leq1) (\vec{a})=\sum_{\substack{0 \leq v_1+v_2+v_3 \leq \frac{p-1}{2} \\ v_1, v_2, v_3 \in \Z_{\geq0}}} \frac{a_1^{2v_1}a_2^{2v_2}a_3^{2v_3}a_{4}^{p-1-2(\sum_{i=1}^3 v_i)}}{({v_1}!{v_2}!{v_3}!)^2(p-1-2(\sum_{i=1}^3 v_i))!}$$ 
where $\vec{a}=(a_1, a_2, a_3, a_4) \in \mathcal{M}_p(\Delta_3)$.\\
\indent \ \  In particular, NP($f$)=HP($\Delta_3$) if and only if $h_p(\Delta_3,\leq1)(\vec{a})\not\equiv 0(\mathrm{mod}\ p)$ where $\vec{a}$ is the vector of coefficients of $f$.
\end{Th}

\begin{proof}
When $n=3$, the non-degenerate Laurent polynomial is given by $$f(x_1,x_2,x_3,x_4)=\sum_{i=1}^3a_i x_{4}\left(x_i+\frac{1}{x_i}\right)+a_{4}x_{4}+\frac{1}{x_4}.$$ where $a_i \in \mathbb{F}_q^*, i=1,2,3,4.$ From Theorem \ref{th212}, HP($\Delta$) has 6 vertices (0,0), (1,0), (5,4), (11,16), (15, 28) and (16, 32).\\
\begin{figure}[ht]
\centering
\includegraphics[scale=0.4]{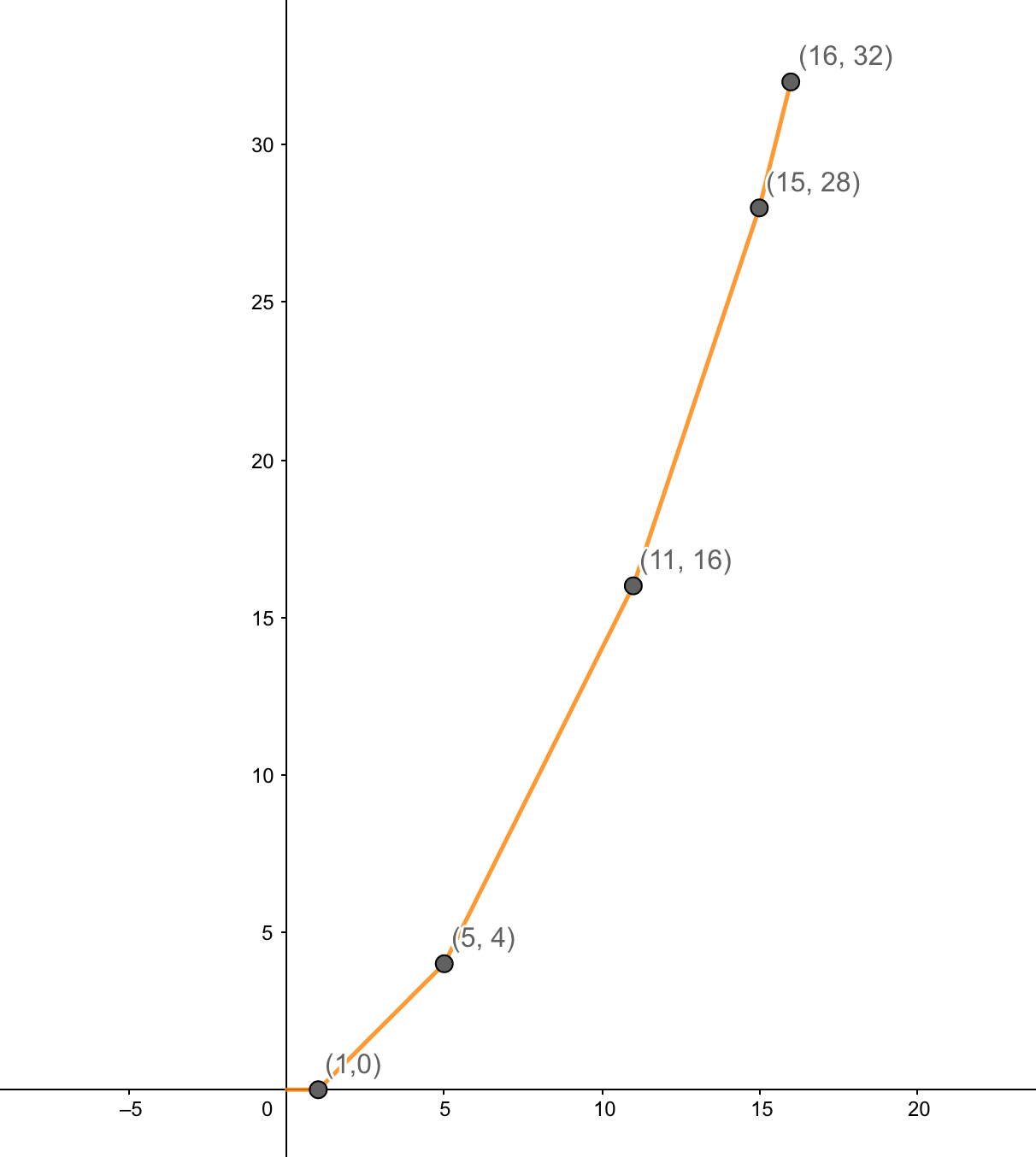}
\caption{Hodge polygon HP($\Delta_3$)}
\label{fig:label}
\end{figure}
\\
\indent \ \ Similar to the proof of Theorem \ref{m1}, let $g=\sum_{i=1}^3a_ix_{4}(x_i+\frac{1}{x_i})+a_{4}x_{4}$ and $\Delta'=\Delta(g)$. We know that NP($f$) coincides with HP($\Delta$) at the $k$-th vertex if and only if NP($g$) coincides with HP($\Delta'$) at the $k$-th vertex. In addition, $h_p(\Delta',0)=\mathrm{det}A_{00}=1\not \equiv 0 (\mathrm{mod} \ p)$ which shows that NP($g$) always pass through (1,0). So NP($f$) coincides with HP($\Delta$) at the first break point (1,0).\\ 
\indent \ \ We claim that NP($f$) is symmetric, i.e., if NP($f$) has a side of a slope $s$ with the horizontal length $l_s$, if and only if it has a side of slope $n+1-s$ with the same horizontal length $l_s$, $s=0,1,\ldots, n+1$. Recall that 
$$L^*(f,T)=\exp  (\sum^\infty _ {k=1} S^*_k (f) \frac{T^k}{k})$$ 
where  $S^*_k(f)=\sum_{x_i \in \F^*_{q^k}} \zeta^{ \mathrm{Tr_k}  f(x_1, \ldots, x_{n+1})}_p$. Then $$L^*(-f,T)=\exp  \left(\sum^\infty _ {k=1} S^*_k (-f) \frac{T^k}{k}\right)=\overline{L^{*}(f,T)}.$$ By Lemma 3.2, this family of L-functions is purely of weight $n+1$, i.e., 
$$L^*(f,T)^{(-1)^n}=\prod^{2^{n+1}}_{i=1}(1-\alpha_i T)$$ where $|\alpha_i|=q^{(n+1)/2}$. So we have 
$$L^*(-f,T)^{(-1)^n}=\prod^{2^{n+1}}_{i=1}(1-\overline{\alpha_i} T)=\prod^{2^{n+1}}_{i=1}\left(1-\frac{q^{n+1}}{\alpha_i} T\right).$$ 
From Lemma 1.3, we know that NP($f$) has a line segment of slope $s$ with horizontal length $l_s$ if and only if NP($-f$) has a line segment of slope $n+1-s$ with length $l_s$. In addition, it's easy to check that NP($f$)=NP($-f$). Consequently, we have the symmetric property: NP($f$) has a side of slope $s$ with the horizontal length $l_s$ if and only if it also has a side of slope $n+1-s$ with the same horizontal length $l_s$.

\indent \ \ Since NP($f$) is symmetric, NP($f$) coincides with HP($\Delta$) at the $k$-th vertex if and only if they coincide at the $(n+1-k)$-th vertex. For $n=3$, NP($f$) and HP($\Delta$) share the same end points (0,0) and (16,32). As proved in Theorem 3.1, we know that NP($f$) matches HP($\Delta$) at point (1,0). By the symmetric property of NP($f$), (15,28) is a also break point on NP($f$). So NP($f$)=HP($\Delta$) if and only if NP($f$) passes through (5,4) if and only if $h_p(\Delta,\leq1) \not \equiv 0(\mathrm{mod} \ p)$. 

\end{proof}

\section{Irreducibility of Hasse Polynomial and Open Problems}
In this section, we focus on the irreducibility of $h_p(\Delta_n,\leq1)$. Recall that a Laurent polynomial $f$ is defined as 
$$f(x_1,\ldots ,x_{n+1})=\sum_{i=1}^na_i x_{n+1}\left(x_i+\frac{1}{x_i}\right)+a_{n+1} x_{n+1}+\frac{1}{x_{n+1}}$$
 where $a_i \in \F^*_{q}$, $i=1,2, \ldots, n+1$. Let $\Delta_n=\Delta(f)$. From previous sections, we know its Hasse polynomial of slope one side $h_p(\Delta_n,\leq1)$ for $n\geq2$ is given by the following formula,
$$h_p(\Delta_n,\leq 1)=\sum_{\substack{0 \leq v_1+\ldots +v_n \leq \frac{p-1}{2} \\ v_1, \ldots v_n \in \Z_{\geq0}}} \frac{a_1^{2v_1}a_2^{2v_2}\ldots a_n^{2v_n}a_{n+1}^{p-1-2(\sum_{i=1}^n v_i)}}{({v_1}!{v_2}!\ldots {v_n}!)^2(p-1-2(\sum_{i=1}^n v_i)!}$$

To study the irreducibility of a polynomial, we first consider the following lemma.
\begin{Lemma}\label{lm41}Assume $n\geq 3$. Let $\mathbb{P}^n$ denote the projective $n$-space over $\bar{\mathbb{F}}_p$ and $H$ be zero locus of a homogeneous polynomial $h$ in $\mathbb{P}^n$, i.e., $H=Z(h)$. Let $\text{Sing}(H)$ denote the set of singular points of $H$. If dim(Sing($H$))$\leq n-3$, then $h$ is irreducible over $\bar{\mathbb{F}}_p.$
\end{Lemma}
\begin{proof}
We give a simple proof by contradiction. Assume $h$ is reducible and $h=h_1h_2$ where $h_1,h_2 \in  \bar{\mathbb{F}}_p[x_1,\ldots, x_{n+1}]$. To get Sing($H$), we compute 
$$\frac{\partial h}{\partial x_i}=\frac{\partial h_1}{\partial x_i}h_2+\frac{\partial h_2}{\partial x_i}h_1$$ where $i=1,2,\ldots, n+1$. So $Z(h_1,h_2)\subseteq \text{Sing}(H)$. Combining $\text{dim}(Z(h_1,h_2))\geq n-2$ and dim($\text{Sing}(H)$)$\leq n-3$, we have $n-2\leq \text{dim}(Z(h_1,h_2)) \leq n-3$ which leads to a contraction.
\end{proof}

Let's first compute $p=3,5$ and 7 as three simple examples.
\begin{Ex} For $p=3$, we have $h_3(\Delta_n,\leq1)=a_1^2+a_2^2+\ldots+a_n^2+\frac{1}{2}a_{n+1}^2$. Since $\text{Sing}(H)=\varnothing$ in $\mathbb{P}^n$, $h_3(\Delta_n,\leq1)=a_1^2+a_2^2+\ldots+a_n^2+\frac{1}{2}a_{n+1}^2$ is irreducible over $\overline{\mathbb{F}}_3$ for $n\geq 2$.
\end{Ex}
\begin{Ex} Suppose $p=5$ and $n\geq 2$. Then $h_5(\Delta_n,\leq1)$ is irreducible over $\overline{\mathbb{F}}_5$.
\begin{proof} 
When $4n+1 \not\equiv 0 \ \mathrm{mod} \ 5$, $\text{Sing}(H)$=$\varnothing$ in $\mathbb{P}^n$ which means $h_5(\Delta_n,\leq1)$ is irreducible. When $4n+1 \equiv 0 \ \mathrm{mod} \ 5$, we have $\text{Sing}(h_5(\Delta_n,\leq1))$ = \{$[\frac{2}{3}:\frac{2}{3}:\ldots:\frac{2}{3}:1]$\}. In this case, $n\geq 5$ and dim(Sing($H$))$= 0 \leq n-3$. By the previous lemma, $h_5(\Delta_n,\leq1)$ is irreducible $\overline{\mathbb{F}}_5$.
\end{proof}
\end{Ex}
\begin{Ex} For $p=7$ and $n\geq 3$, $h_7(\Delta_n,\leq1)$ is irreducible over $\overline{\mathbb{F}}_7$.
\end{Ex}

\begin{proof}
Since $h_7(\Delta_n,\leq1)$ is very complicated, we first consider $n=3$ and then generalize to $n\geq 3$. 
In $\mathbb{P}^3$, it's easy to check that $\text{Sing}(h_7(\Delta_3,\leq1))$=$\{[a_1:a_2:a_3:a_4] $ $| 2a_1^2=2a_2^2=2a_3^2=a_4^2 \}$ $\cup$ $\{[0:0:a_3:a_4]$|$ 4a_3^2=a_4^2 \}$ $\cup$ $\{[0:a_2:0:a_4] $ $| 4a_2^2=a_4^2 \}$ $\cup$ $\{[a_1:0:0:a_4] | 4a_1^2=a_4^2 \}$. So dim(Sing($H$))$= 0 \leq n-3$ for $n = 3$ which means $h_7(\Delta_3, \leq1)$ is irreducible over $\bar{\mathbb{F}}_7$. 

For $n\geq 3$, we prove by contradiction. Assume $n\geq 4$, then we have $h_7(\Delta_n,\leq1)\equiv \bar{h}_n(a_1,a_2, a_3, a_4) \ \text{mod} (a_5, \ldots, a_{n+1})$. If $h_7(\Delta_n,\leq1)$ is reducible over $\bar{\mathbb{F}}_7$, i.e., $h_7(\Delta_n,\leq1)=g_n l_n$, then $\bar{h}_n \equiv \bar{g}_n\bar{l}_n  \ \text{mod} (a_5, \ldots, a_{n+1})$ which means $\bar{h}_n$ is reducible. From formula (\ref{eq31}), $h_7(\Delta_n,\leq1)\equiv h_7(\Delta_3,\leq1) \ \text{mod} (a_5, \ldots, a_{n+1}).$ Since $h_7(\Delta_3,\leq1)$ is proved to be irreducible over $\bar{\mathbb{F}}_7$, this leads to a contradiction.
\end{proof}
Based on these examples, we give our hypothesis.

\begin{Conj} Assume $n\geq 3$. Let $p$ be an odd prime and $f$ be a Laurent polynomial in that family with $\Delta_n=\Delta(f)$. The Hasse polynomial of the slope one side $h_p(\Delta_n,\leq1)$ is irreducible over $\bar{\F}_p$. 
\end{Conj}

To prove the conjecture, it's sufficient to check if $h_p(\Delta_3,\leq1)(\vec{a})$ is irreducible over $\bar{\F}_p$. For the proving of the conjecture, we leave it as an open problem.

\nocite{*}
\bibliographystyle{amsalpha}
\bibliography{ref}

\end{document}